\documentclass[a4paper, 12pt]{article}

\usepackage{amssymb}
\usepackage{color}
\usepackage{amsmath}
\usepackage{graphicx}
\usepackage{textcomp}
\usepackage[russian]{babel}

\newcommand{\RR}{\mathbb{R}}

\newcommand{\rbr}[1]{\left(#1\right)}

\newcommand{\floorbr}[1]{\left\lfloor #1\right\rfloor}

\newcommand{\fbr}[1]{\left\{#1\right\}}

\newcommand{\ra}{\rightarrow}
\newcommand{\Ra}{\Rightarrow}

\newcommand{\half}[1][1]{\frac{#1}{2}}
\newcommand{\Array}[2]{\begin{array}{#1}#2\end{array}}
\newcommand{\alignR}[1]{\begin{flushright}#1\end{flushright}}

\newcommand{\eqsys}[2]{\left\{\Array{l}{#1\\#2}\right.}
\newcommand{\bproof}{

\smallskip\noindent $\blacktriangleright$ }

\newcommand{\eproof}{\alignR{$\blacktriangleleft$}}
\newcommand{\eps}{\varepsilon}
\newcommand{\ds}{\displaystyle}
\newcommand{\sign}{\mathop{\rm sign}\nolimits}%функция sign
\newcommand{\tr}{\kappa}

\newcommand{\uplim}{\mathop{\overline{\lim}}\limits}
\newcommand{\downlim}{\mathop{\underline{\lim}}\limits}
\newcommand{\llim}[1]{\lim\limits_{#1\ra\infty}}
\newcommand{\upllim}[1]{\uplim_{#1\ra\infty}}
\newcommand{\downllim}[1]{\downlim_{#1\ra\infty}}
\renewcommand{\phi}{\varphi}
\newcommand{\teta}{\theta}

\newcommand{\defst}[1]{\textit{#1}}

\newtheorem{theorem}{Теорема}

\newtheorem{lemma}{Лемма}
\newtheorem{Lemma}{Основная лемма}

\newcommand{\omlen}{L_{\Omega}}
\newenvironment{proof}{\bproof}{\eproof}

\begin{document}
%\small{
\begin{center}
\large{\textbf{Оценка снизу скорости блуждания решения линейного дифференциального уравнения третьего порядка через частоту нулей}}

\vspace{7mm} 

\large{\textbf{Тихомирова А.В.}}

\end{center}

В работе сравниваются две характеристики решений линейных дифференциальных уравнений 3-го порядка с переменными коэффициентами. %, а именно скорость блуждания и частота нулей решения дифференциального уравнения.
Доказано, что для решения на конечном отрезке длина проекции фазовой кривой на единичную сферу оценивается снизу через число нулей решения. В предельном случае, когда длина отрезка стремится к бесконечности, рассматриваются соответствующие средние величины: скорость блуждания и частота нулей, и для их отношения получена точная нижняя оценка.

\subsection{Описание исследуемых величин}
Рассмотрим линейное дифференциальное уравнение $n$-го порядка
%Рассмотрим систему $A\in \ M^n$:
%\[\dot{x}=\ \ A\left(t\right)x, \ \ \ x\in {{\mathbb R}}^n,\ \ \ t\in {{\mathbb R}}^+\equiv \left[0;\infty \right.),\] 
%а если 
%\[A(t)=\left( \begin{array}{cccc}
%0 & 1 & \cdots  & 0 \\ 
%\vdots  & \vdots  & \ddots  & \vdots  \\ 
%0 & 0 & \dots  & 1 \\ 
%-a_{n\left(t\right)} & -a_{n-1}\left(t\right) & \cdots  & -a_1\left(t\right) \end{array}
%\right)\] 
%то система эквивалентна уравнению $a\ \in \ {{\mathcal E}}^n:$
\begin{equation}
y^{(n)}+a_1\left(t\right)y^{\left(n-1\right)}+\dots +a_{n-1}\left(t\right)\dot{y}+a_n\left(t\right)y=0,\  t\in \ {{\mathbb R}}^+,
\label{ur1}
\end{equation}
с переменными коэффициентами $a_1,...,a_n\in C(\RR^+)$.
%при замене 
%\[x={\psi }^ny\equiv \left(y,\dot{y},\cdots ,y^{(n-1)}\right).\] \\
%Пусть задан вектор и, соответственно, вектор-функции
%\[m\equiv \left(m_1, \ \dots \ ,\ m_n\right)\in {{\mathbb R}}^n,\ x=\left(x_1,\ \dots \ ,x_n\right):{{\mathbb R}}^+\to \ {{\mathbb R}}^n\] 
Обозначим  через $\nu(y,t)$ число нулей решения $y$ на промежутке $\left[0,t\right]$.

\bigskip
\noindent{\underbar{Определение 1.}}  Назовем \defst{нижней} (соотв. \defst{верхней}) \defst{частотой нулей} решения $y$ величину
$$\check{\nu}(y)= {\mathop{\underline{\lim }}_{t\to \infty} \frac{\pi }{t}\nu(y,t)}\quad\rbr{\mbox{соотв. }\hat{\nu}(y)= {\mathop{\overline{\lim }}_{t\to \infty } \frac{\pi}{t}\nu(y,t)}}.$$ 
Если эти пределы равны, то их общее значение $\nu(y)$ назовем \defst{точной частотой нулей} решения $y$.
%$\exists\lim\limits_{t\to\infty}\frac{\pi }{t}\nu \left(x,m,t\right)$, то $\check{\nu}(x,m)=\hat{\nu}(x,m)=:\nu(x,m)$

%Далее везде будем считать, что $m=\left(1,\ \ 0,\ \ \cdots ,\ 0\right).$  Заметим, что в этом случае $\check{\nu }\left(x,m\right)$ - частота нулей самого решения $y$ уравнения \eqref{ur1}, или, что то же, первой координаты решения $x$ эквивалентной системы $A\in {{\mathcal E}}^n$.\\

\bigskip
\noindent {\underbar{Определение 2.}} Назовем \defst{нижней} (соотв. \defst{верхней}) \defst{скоростью блуждания} вектор-функции $x:{{\mathbb R}}^+\to {{\mathbb R}}^n$ величину
\begin{equation}
\check{\mu}(x)=\mathop{\underline{\lim }}_{t\to \infty } \frac{1}{t}\gamma(x,t)\quad\rbr{\mbox{соотв. }\hat{\mu }(x)={\mathop{\overline{\lim }}_{t\to \infty }\frac{1}{t}\gamma(x,t)}},
\label{ur2}
\end{equation}
где
$$\gamma(x,t)=\int\limits^t_0{\left|\frac{d}{d\tau}\left(\frac{x\left(\tau \right)}{|x(\tau)|}\right)\right|}\,d\tau ,\quad |x|=\sqrt{x_1^2+...+x_n^2}.$$ \\
Если для функции $x$ пределы \eqref{ur2} равны, то их общее значение $\mu(x)$ назовем \defst{точной скоростью блуждания} вектор-функции $x$.
%Если $\exists\lim\limits_{t\to\infty}\frac{1}{t}\gamma \left(x,t\right)$, то $\check{\mu}(x,m)=\hat{\mu}(x,m)=:\mu(x,m)$

Геометрическая интерпретация величины $\gamma(x,t)$:  $\gamma(x,t)$ --- это длина траектории единичного вектора $\frac{x(\tau)}{|x(\tau)|}$ за время от 0 до $t$.

Нас будут интересовать величины $\gamma(x,t)$ и $\mu(x)$ в случае равенства $x = \left(y,\dot{y},...,y^{(n-1)}\right)$, где $y$ --- решение уравнения \eqref{ur1}. В этом случае будем обозначать:
\begin{equation}
\gamma_n(y,t)=\gamma(x,t),\quad
\check{\mu}_n(y)=\check{\mu}(x),\quad
\hat{\mu}_n(y)=\hat{\mu}(x),
\label{ur3}
\end{equation}

В работе \cite{sergeev} рассматривались более общие функционалы решений, в частности,
$$ %\hat{\sigma}(x)=\inf_{m\in\RR^n\setminus\{0\}}\hat{\nu}(\sum m_i x_i),\quad
\hat{\zeta}(x)=\uplim_{t\ra\infty}\frac{\pi}{t}\inf_{m\in\RR^n\setminus\{0\}}\nu(\sum_{i=1}^n x_i m_i,t)
\quad\rbr{\check{\zeta}(x)=\downlim_{t\ra\infty}\frac{\pi}{t}\inf_{m\in\RR^n\setminus\{0\}}\nu(\sum_{i=1}^n x_i m_i,t)},$$
$$ %\hat{\rho}(x)=\inf_{L\in Aut(R^n)}\hat{\mu}(Lx),\quad
\hat{\eta}(x)=\uplim_{t\ra\infty}\frac{1}{t}\inf_{L\in Aut(R^n)}\gamma(Lx,t)
\quad\rbr{\check{\eta}(x)=\downlim_{t\ra\infty}\frac{1}{t}\inf_{L\in Aut(R^n)}\gamma(Lx,t)},$$
и для них были получены оценки:
$$\check{\eta}(x)\ge\check{\zeta}(x),\qquad\hat{\eta}(x)\ge\hat{\zeta}(x),$$
не зависящие от $n$.
Изменяя допустимые области значений автоморфизма $L$ и направления $m$, получаем другие варианты величин, для которых можно тоже ставить вопрос об их сравнимости.

В настоящей работе исследуется случай, когда пространство трехмерно, $L$ всегда тождественен, а $m\equiv(1,0,0)$. Приведенные доказательства существенно используют, что пространство трехмерно, и поэтому результаты относятся только к трехмерному случаю.

%\noindent Пусть при изменении времени  $t\in \ {{\mathbb R}}^+$ каждый луч с началом в нуле, натянутый на вектор $x\left(t\right)\in \ {{\mathbb R}}^n$ оставляет свой след $e_x(t)$  на единичной сфере в ${{\mathbb R}}^n$, описывая в ней некоторую траекторию. Тогда величина $\gamma \left(x,t\right)$ -  это длина пути следа $e_x(t)$  на единичной сфере за время от $0$ до $t$.

\subsection{Формулировка результатов}
В данной работе рассматривается случай $n=3$, когда \eqref{ur1} имеет вид
$$y'''(t)+a(t)y''(t)+b(t)y'(t)+c(t)y(t)=0, \quad a,b,c\in C(\RR^+).\eqno (1')$$

\bigskip

\noindent Обозначим
\begin{tabbing}
$x$\hspace{13pt} \=$ = (x_0,x_1,x_2) = \rbr{y,\dot{y},\ddot{y}}$ --- точка в фазовом пространстве,\\
$\Omega_+$\hspace{7pt} \= --- область $\{x_0x_2-x_1^2>0, x_2>0\}$ на сфере $|x|=1$ в фазовом пространстве\\
$\partial\Omega_+$\> --- граница области $\Omega_+$ на единичной сфере,\\
$\omlen$\> --- длина $\partial\Omega_+$ (абсолютная постоянная). В Добавлении показано, что 
\end{tabbing}
$$\omlen = 4\int\limits_{0}^{\pi}\frac{\sqrt{5 - \cos\alpha}}{7 + \cos\alpha}\,d\alpha % =6\mathop{\Pi}\Bigl(\frac{1}{4}\Bigm|-\frac{1}{2}\Bigr)-4\mathop{K}\Bigl(-\frac{1}{2}\Bigr)
 .$$

%\\$\Omega_-$--- область на сфере в фазовом пространстве $\{\ddot{y}y-\dot{y}^2>0, \ddot{y}<0\}$\\
%Таким образом, $\Omega=\Omega_+\sqcup\Omega_-$
\begin{theorem}
Для любого ненулевого решения $y$ уравнения $(1')$ и любого $t>0$ верна оценка
$$\gamma_3(y,t)>\half[\nu(y,t)-5]\omlen. %, \qquad \mbox{где} \quad x=\rbr{y,\dot{y},\ddot{y}}
$$
\end{theorem}

\begin{theorem}$ $
\begin{enumerate}
\item[а)] Для любого ненулевого решение  $y$ уравнения $(1')$ верны оценки $$\ds\hat{\mu}_3(y)\ge\frac{\omlen}{2\pi}\hat{\nu}(y),\quad %\mbox{ и }
 \ds\check{\mu}_3(y)\ge\frac{\omlen}{2\pi}\check{\nu}(y).$$ % Такое же неравенство верно для $\check{\mu}$ и $\check{\nu}$.
\item[б)] Для любого $\eps>0$ существуют гладкие функции $a$, $b$, $c$ и $y$, связанные равенством $(1')$, такие что $\mu_3(y)<\rbr{\frac{\omlen}{2\pi}+\eps}\nu(y)$.
\end{enumerate}
\end{theorem}

\newcommand{\vtr}{\tilde{\tr}}
\newcommand{\vfi}{\tilde{\phi}}
\newcommand{\vteta}{\tilde{\teta}}
\newcommand{\mdpi}{\ (\mathrm{mod}\ 2\pi)}

\subsection{Доказательства.}

В случае $n=3$ фазовая кривая примет вид $x(t)=\rbr{y(t),y'(t),y''(t)}$. % % %
Обозначим $\ds\tr(t):=\frac{x(t)}{|x(t)|}$~--- проекция фазовой кривой на единичную сферу. Если решение ненулевое ($x(t)\not\equiv 0$), то по теореме единственности $\ds\forall t\ x(t)\ne 0\ \Ra\ |x(t)|>0\ \Ra\ \tr(t)$ определено в любой точке, т.е. определение корректно.

Введем следующие обозначения:
\\$\Omega$ --- область $\{x_2x_0-x_1^2>0\}$ на сфере $|x|=1$ в фазовом пространстве ;
\\$\Omega_-$ --- часть области $\Omega$, где $x_2<0$.\\
Таким образом, $\Omega=\Omega_+\sqcup\Omega_-$.\\
В дальнейшем будем отождествлять $x(t)$ с вектором $\left(y(t),\dot{y}(t),\ddot{y}(t)\right)$. Допустим, что $y$ и $\dot{y}$ одновременно не обращаются в 0.
Перейдем к сферической системе координат:

 $$\left\{\Array{l}{\ds y=|x(t)|\cos\teta\cos\varphi,\\
\ds\dot{y}=|x(t)|\cos\teta\sin\varphi,\\
\ds\ddot{y}=|x(t)|\sin\teta.}\right.$$
Выразим $\dot{\varphi}$ через $y$, $\dot{y}$ и $\ddot{y}$:

\begin{lemma}
\begin{equation}
\dot{\varphi} = \frac{\ddot{y}y-\dot{y}^2}{y^2+\dot{y}^2}.
\end{equation}
\end{lemma}
\begin{proof}
Для тех точек, где $y\ne 0$ выполнено $\frac{\dot{y}}{y}=\tg\varphi\ \Ra$
 $$\Ra\ \dot{\rbr{\frac{y'}{y}}} = \frac{\dot{\varphi}}{\cos^2\varphi}=\dot{\varphi}(1+\tg^2\varphi)=\dot{\varphi}\rbr{1+\frac{{\dot y}^2}{y^2}}$$
$$\dot{\varphi}=\frac{\frac{\ddot{y}y-\dot{y}^2}{y^2}}{1+\frac{{\dot y}^2}{y^2}}=\frac{\ddot{y}y-\dot{y}^2}{y^2+\dot{y}^2}$$
Аналогично, для точек, где $\dot{y}\ne 0$, выполнено
$\ds\frac{y}{\dot{y}}=\ctg\varphi\ \Ra$
$$\Ra\ \dot{\rbr{\frac{y}{y'}}} = -\frac{\dot{\varphi}}{\sin^2\varphi}=-\dot{\varphi}(1+\ctg^2\varphi)=-\dot{\varphi}\rbr{1+\frac{y^2}{{\dot y}^2}};$$
$$\dot{\varphi} = -\frac{\frac{\dot{y}^2-\ddot{y}y}{\dot{y}^2}}{1+\frac{y^2}{{\dot y}^2}} = \frac{\ddot{y}y-\dot{y}^2}{y^2+\dot{y}^2}.$$
Мы предположили, что $\forall t\ y(t)\ne 0\mbox{ или }\dot{y}(t)\ne 0$, утверждение леммы всюду выполнено.
%\begin{equation}
%\dot{\varphi} = \frac{\ddot{y}y-\dot{y}^2}{y^2+\dot{y}^2}.
%\end{equation}
\end{proof}

Отсюда $\dot{\varphi}(t)>0\ \Leftrightarrow\ \ddot{y}(t)y(t)-\dot{y}^2(t)>0\ \Leftrightarrow\ \tr(t)\in\Omega$.

Найдем границу области $\Omega$ в сферических координатах.
$$0<\ddot{y}(t)y(t)-\dot{y}^2(t)=|x(t)|^2\sin\teta\cos\teta\cos\varphi - |x(t)|^2(\cos\teta\sin\varphi)^2.$$
Вне полюсов $\cos\teta>0$, поэтому поделим неравенство на положительное число $|x(t)|^2\cos^2\teta$. Тогда: $\dot{\varphi}>0\ \Leftrightarrow\ \sin\teta\cos\phi-\cos\teta\sin^2\phi>0$, %поэтому поделим неравенство на $\cos\teta$ и получим, что область $\Omega$ в сферической системе имеет вид: $\sin\teta\cos\phi-\cos\teta\sin^2\phi>0$.
и граница  $\Omega$ будет иметь вид:
\begin{equation}
\teta = \teta_0(\phi):=\arctan\frac{\sin^2\phi}{\cos\phi}, \phi\not\equiv\half[\pi](\mathop{\mathrm{mod}} \pi).
\end{equation}

Заметим, что при $y=0$, $\dot\phi=\frac{-\dot{y}^2}{\dot{y}^2}=-1$, поэтому большая окружность(пересечение сферы и плоскости $y=0$) не пересекает область $\Omega$. В сферических координатах $y=0\ \Leftrightarrow\ \cos\phi=0\ \Leftrightarrow\ \phi\equiv\half[\pi] (\mathop{\mathrm{mod}} \pi)$.

Отсюда видно, что аналогично 2-мерному случаю между двумя нулями функции $y$ величина $\varphi$ должна уменьшиться на $\pi$ (расстояние между 2-мя соседними нулями косинуса). Причем внутри области $\Omega$ величина $\varphi$ увеличивается, а вне --- уменьшается.
Обозначим за $P(\phi,\teta)$ точку на сфере со сферическими координатами $\phi$ и $\teta$

Введем еще одно обозначение: $\Omega_-$--- область на сфере в фазовом пространстве $\{\ddot{y}y-\dot{y}^2>0, \ddot{y}<0\}$
Таким образом, $\Omega=\Omega_+\sqcup\Omega_-$
\newcommand{\zs}{\pi_\Omega}

\begin{lemma}\label{lemOmega}
Область $\Omega$ состоит из двух выпуклых областей $\Omega_+$ и $\Omega_-$, центрально симметричных относительно 0 и зеркально симметричных относительно плоскости $\zs$: $\ddot{y}+y=0$, причем полюсы введенной сферической системы координат лежат на ее границе.
\end{lemma}
\begin{proof}

На сфере определение выпуклости эквивалентно следующему: множество $M$ на сфере выпукло, если множество $\fbr{\alpha x:x\in A, \alpha\ge 0}$~--- выпукло в $\RR^3$ (т.е. множество, состоящее из лучей, проведенных из центра сферы во все точки множества $A$).
Для области $\Omega$ это все $(y,\dot{y},\ddot{y}): y\ddot{y}-\dot{y}^2>0$
Проверим, что это внутренность эллиптического конуса, для этого сделаем линейную замену координат: 
$$\eqsys{y=\frac{u+v}{\sqrt{2}},}{\ddot{y}=\frac{u-v}{\sqrt{2}}.}$$
Тогда та же область будет иметь вид в новых координатах: $\ds \frac{u^2-v^2}{2}-\dot{y}^2>0\ \Leftrightarrow\ u^2=v^2+2\dot{y}^2$, что является внутренностью эллиптического конуса. Эллиптический конус состоит из двух выпуклых частей, центрально симметричных относительно $0$ и зеркально симметричных относительно плоскости $u=0$, или, делая обратную замену, плоскости $y+\ddot{y}=0$. Здесь важно, что замена ортогональная, поэтому зеркальная симметрия при ней сохраняется. Такой же симметрией обладает и пересечение конуса со сферой, что и требовалось доказать.
Обозначим за $\Omega_+$ ту часть, где $u>0$ и $\Omega_-$~--- где $u<0$. Для полюсов, т.е. точек с координатами $(0,0,\pm 1)$ выполняется равенство $0c-0^2=0\ \Ra\ $ такие точки находятся на границе области $\Omega$.
\end{proof}

\newcommand{\cO}{\Omega}
\begin{lemma}\label{deletePart}
Пусть наша кривая
$\tr(t)=P(\phi(t),\theta(t)), t\in[t_0,t_1]$ не проходит через полюсы сферы и  $\eqsys{\phi(t_0)\equiv\phi(t_1)\equiv \half[\pi]\mdpi}{\varphi(t_1)=\varphi(t_0)-2\pi}$, т.е. она делает один оборот вокруг полюса по часовой стрелке.

Тогда из нее можно выбросить некоторые части так, чтобы осталась непрерывная кривая $\vtr: [0;1]\ra S^2\setminus\Omega$, концы которой совпадают с концами $\tr$, т.е. $\vtr(0)=\tr(t_0),\vtr(1)=\tr(t_1)$.
\end{lemma}
\begin{proof}\\
\includegraphics[width=\linewidth]{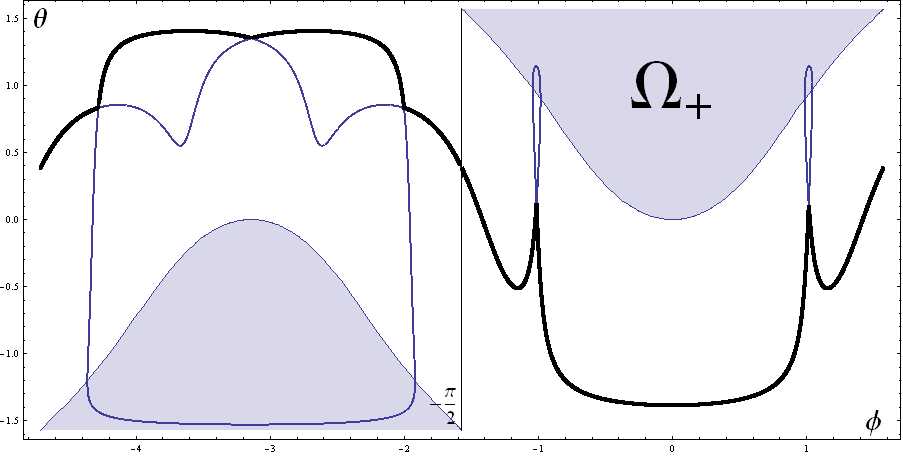}
\newcommand{\T}{\tilde{t}}
\newcommand{\dO}{\partial\Omega}

{\bf I.} Сначала покажем, что $\forall a\in[\phi(t_1);\phi(t_0)]\ \exists t_a\in(t_0;t_1): \phi(t)=a \mbox{ и } \dot{\phi}(t)\le 0$:\\
Пусть $ t_a:= \sup\fbr{t:\phi(t)\ge a}$. Тогда $\phi(t_a+h)<a$ при $h>0$, значит из непрерывности $\phi(t)$ сразу следует, что
$$\eqsys{\phi(t_a)\ge a,}{\ds \phi(t)=\lim\limits_{h\ra +0} \phi(t_a+h)\le a}\ \Ra\ \phi(t_a)=a;$$
%$\phi(t)=\lim\limits_{h\ra +0} \phi(t+h)\ge a$, значит $\phi(t)=a $.
$$\dot{\phi}(t_a)=\lim\limits_{h\ra 0}\frac{\phi(t_a+h)-\phi(t_a)}{h} = \lim\limits_{h\ra+0}\frac{\phi(t_a +h) - a}{h} \le 0.$$
Поскольку $\dot{\phi}(t_a)\le 0$, то $\tr(t_a)\notin\Omega$.
\newcommand{\tm}{t_{1/2}}
Также можно заметить, что $\exists ! \tm: \phi(\tm)=\phi(t_0)-\pi\equiv-\half[\pi]\mdpi \Ra \dot{\phi}(\tm)<0$, значит через "прямую" $\phi=\phi(\tm)$ можно перейти только в одну сторону, а значит только один раз (на это прямая $-\half[\pi]$).

 Поэтому достаточно выделить часть кривой $ \tr $, соединяющую $\tr(t_0)$ и $\tr(\tm)$, не пересекающуюся с $\Omega$, а потом часть, соединяющую $\tr(\tm)$ и $\tr(t_1)$. При этом поиск второй части аналогичен первой, достаточно поменять знак $t$ и поменять $-t_0$ и $-t_1$ местами.
Поэтому осталось доказать, что из части
$\fbr{\tr(t):t\in[t_0;\tm]}\setminus \Omega$ можно выделить кривую, соединяющую $\tr(t_0)$ и $\tr(\tm)$.

{\bf II.} Будем теперь рассматривать все на отрезке $[t_0;\tm]$. Здесь $\phi(t)\in[\phi(t_0)-\pi,\phi(t_0)]$, поэтому область $\Omega$ имеет вид $\half[\pi]>\teta>\teta_0(\phi)$ (граница области $\Omega$).

Обозначим $M:=\fbr{\tr(t):t\in[t_0;\tm]}\setminus\Omega, M_a := \fbr{\teta(t):\phi(t)=a,\tr(t)\notin\Omega}$ -- выпускаем из полюса сферы лучи под углом $a$ и пересекаем его с множеством $M$. (А на рисунке $M_a$  --- пересечение вертикальной линии с графиком к вне области $\Omega$)

 $t_a\in M_a$(по определению $t_a$) $\Ra M_a$ непусто. 
Из того, что $\tr(t)$ пересекает полуокружности $\phi=\phi(t_0)$ и $\phi=\phi(t_0)-\pi$ по одному разу следует, что
\begin{equation}
M_{\phi(t_0)}=\fbr{\teta(t_0)},M_{\phi(\tm)}=\fbr{\teta(\tm)}. \label{M0m}
\end{equation}

{\bf III.} Рассмотрим функцию $\Theta(a) := \min M_a$ на отрезке $[\phi(\tm),\phi(t_0)]$. (На рисунке это выделено жирной линией). Тогда $P(a,\Theta(a))\in M$.  Из \eqref{M0m} получаем, что $\Theta(\phi(t_0))=\teta(t_0),\Theta(\phi(\tm))=\teta(\tm)$, поэтому
$$\Array{l}
{P(\phi(t_0),\Theta(\phi(t_0)))=P(\phi(t_0),\teta(t_0))=\tr(t_0),\\
P(\phi(\tm),\Theta(\phi(\tm)))=P(\phi(\tm),\teta(\tm))=\tr(\tm).}$$

Осталось показать, что $\Theta(a)$ непрерывна на своей области определения. Допустим, $a_0$~--- точка разрыва. Тогда есть две возможности:
\begin{itemize}
\item  $\Theta(a_0)> y_0:=\downlim_{a\ra a_0}\Theta(a)=\llim{n}\Theta(a_n)$ для некоторой последовательности ${a_n\ra a}$.
Поскольку $M=[\tr]\setminus \Omega$~--- замкнутое множество, а $P(a_n,\Theta(a_n))\in M$, то $P(a_0,y_0)=\llim{n}P(a_n,y_n) \in M$, поэтому $y_0\in M_{a_0}\ \Ra\ \Theta(a_0)\le y_0$~--- противоречие.
\item $\Theta(a_0)<y_0:=\uplim_{a\ra a_0}\Theta(a)=\llim{n}\Theta(a_n)$ для некоторой последовательности ${a_n\ra a}$.
Аналогично предыдущему случаю, $P(a_0,y_0)=\llim{n}P(a_n,y_n) \in M$, поэтому $y_0\le\teta_0(a_0)\ \Ra\ \Theta(a_0)<\teta_0(a_0)$.

$\Theta(a_0) = \teta(t),a_0=\phi(t)$ для некоторого $t$. При этом $\teta(t)<\teta_0(\phi(t))\ \Ra\ \tr(t)\notin\partial\Omega\ \Ra\ \dot{\phi}(t)>0$. Тогда по теореме об обратной функции
$$\exists t(a),\delta>0 : t(a_0)=t,\forall a:|a-a_0|<\delta\ \Ra\ \phi(t(a))=a.$$
 Из непрерывности $\teta(t(a))$, 
$$\exists 0<\delta_1\le\delta: \forall a:|a-a_0|<\delta_1\Ra |\underbrace{\teta(t(a))}_{\in M_a}-\underbrace{\teta(t(a_0))}_{\Theta(a_0)}|<\half[y_0-\Theta(a_0)].$$
Отсюда при $|a-a_0|<\delta_1$ $\min M_a\le \teta(t(a))\le \Theta(a_0)+\half[y_0-\Theta(a_0)]$, значит $\uplim_{a\ra a_0}\le \Theta(a_0)+\half[y_0-\Theta(a_0)]<y_0$~--- противоречие.
\end{itemize}
Значит $\Theta(a)$ непрерывна, а следовательно, кривая $P(a,\Theta(a))$ и есть искомая.

 Параметризацию всегда можно выбрать так, чтобы новая кривая была определена на отрезке $[0;1]$. Это всегда можно сделать масштабированием. Длина кривой от параметризации не зависит, а нас в конечном итоге интересует ее длина.
\end{proof}

Пусть есть точка $M$ на сфере и кривая $\gamma:[0;1]\ra S$, причем точки $M$ и противоположная ей не лежат на $\gamma$. Определим индекс точки $M$ на сфере относительно кривой $\gamma$.
Введем сферическую систему координат с полюсом в точке $M$. 

\begin{lemma}\label{oneCurve}
Существует кратчайшая замкнутая кривая $l:[0;1]\ra S\setminus\Omega$, для которой в сферических координатах 
\smallskip
\\$\phi(0)=0,\phi(1)=2\pi n$ (кривая делает $n$ оборотов вокруг полюса),
\smallskip
\\$\teta(0)=\teta(1)$ (кривая замкнута).
\end{lemma}
\begin{proof}
Рассмотрим множество $K:=\fbr{(\phi,\teta)\in \RR^2: P(\phi,\teta)\in S\setminus\Omega}$. На этом множестве введем риманову метрику, соответствующую метрике на сфере: $g_{ij}(\phi,\teta)=\left(\Array{cc}{\cos^2\teta&0\\0&1}\right)$. Рассмотрим $K$, как метрическое пространство с функцией расстояния $\rho_S$, порожденной метрикой $g_{ij}$:
$$\rho_S(x,y):= \inf\limits_{\gamma:[0;1]\ra K:\gamma(0)=x,\gamma(1)=y} \int\limits_{0}^{1}\sqrt{ \phi'^2(s)\cos^2\teta+\teta'^2(s)}\ds.$$
 Все точки $x,y:\rho_S(x,y)=0$ будем считать совпадающими. Это те точки, которые соответствуют полюсам сферы, т.е. где $|\teta|=\half[\pi]$ и $\phi$ отличается не более, чем на $\pi$. Тогда $K$ --- замкнутое локально компактное метрическое пространство. Отсюда по следствию из  теоремы Хопфа-Ринова \cite[теор.10.9]{hr} существует кратчайшая, соединяющая любые 2 точки из $K$. То есть $\inf$ в определении $\rho_S$ на самом деле достигается. 
 
 Отображение $(\phi,\teta)\ra P(\phi,\teta)$, переводящее $K$ в $S\setminus\Omega$, по построению является локально изометричным, поэтому длины кривых сохраняются. А значит длина кривой $\gamma$, заданной в сферических координатах $\gamma(s)=P(\phi(s),\teta(s)$, будет равна длине кривой $(\phi(s),\teta(s))$ в множестве $K$. Поэтому поиск кратчайшей замкнутой кривой сводится к поиску кривой наименьшей длины в $K$, с теми же условиями, то есть
\begin{equation}\label{usll}
\phi(0)=0,\phi(1)=2\pi n,\teta(0)=\teta(1).
\end{equation}
 Длина любой кривой, удовлетворяющей этим условиям, не менее $\rho_S((0,\teta(0)),(2\pi n,\teta(1)))=\rho_S((0,a),(2\pi n,a)),\teta(0)=\teta(1)=a$.

Пусть $L(a):=\rho_S((0,a),(2\pi n,a)),|a|\le\half[\pi]$. Тогда по неравенству треугольника
$$L(b)\le\underbrace{\rho_S((0,b),(0,a))}_{\le|b-a|}+\underbrace{\rho_S((0,a),(2\pi n,a))}_{L(a)} +\underbrace{\rho_S((2\pi n,a),(2\pi n,b))}_{\le|b-a|}\le 2|b-a|+L(a).$$
Отсюда $|L(b)-L(a)|\le 2|b-a|$, откуда следует, что функция $L$ непрерывна. Значит она достигает своего минимума на отрезке $[-\half[\pi];\half[\pi]]$ в некоторой точке $a_0$. Пусть $l_0$ --- кратчайшая, соединяющая $(0,a_0)$ и $(2\pi n,a_0)$. Тогда она удовлетворяет условиям \eqref{usll} и по построению ее длина, равная $L(a_0)<L(a)$, не превосходит любой другой кривой, удовлетворяющей \eqref{usll}. Отображение $(\phi,\teta)\ra P(\phi,\teta)$ кривой $l_0$ на сферу и будет искомой кривой $l$.
\end{proof}

Поскольку кривая замкнута, то условия на кривую эквивалентны следующим:
\\$|\phi(t_0)-\phi(t_1)|=2\pi n$,
\\$\teta(t_0)=\teta(t_1)$.\\
для некоторых $t_1>t_0$. Заменой параметризации и начальной точки кривой, этот случай сводится к рассмотренному в лемме \ref{oneCurve}.

\begin{lemma}\label{KLength}
Пусть кривая $\vtr(t)=P(\phi(t),\teta(t)),t\in[t_0;t_1]$ не пересекает $\Omega$ и $\phi(t_1)=\phi(t_0)-2\pi N$, не проходит через полюсы $(\teta(t)\ne \pm \half[\pi])$ и каждое значение $\phi$ принимает ровно 1 раз.

Тогда ее длина не менее, чем $N\omlen-|\teta(t_1)-\teta(t_0)|$.
\end{lemma}
\begin{proof}

\newcommand{\ntr}{g}

Сведем задачу к замкнутой кривой на полусфере, обходящей $N$ раз выпуклую кривую $\partial\Omega$. Для этого сначала добавим к кривой дугу большой окружности, соединяющей $\vtr(t_0)$ и $\vtr(t_1)$; т.к. $\phi(t_1)=\phi(t_0)-2\pi N\equiv \phi(t_0)\mdpi$, то длина этой дуги равна $|\teta(t_1)-\teta(t_0)|$. Обозначим полученную замкнутую кривую $\tilde{l}$. Пусть $l$ --- кратчайшая замкнутая кривая с условиями $|\phi(t_0)-\phi(t_1)|=2\pi n,\teta(t_0)=\teta(t_1)$, существующая по лемме 3. Тогда длина $\tilde{l}$ не меньше, чем длина $l$.

 Отразим теперь часть сферы $\ddot{y}+y<0$ относительно плоскости $\ddot{y}+y=0$. При этом $\Omega_-$ перейдет в $\Omega_+$ по лемме \ref{lemOmega}. Такая симметрия~--- непрерывное преобразование, поэтому наша замкнутая кривая $l$ перейдет в замкнутую кривую, которую обозначим за $\ntr(t)$. Ее длина будет равна длине $l$, поскольку ее часть, отличная от $l$, получена из $l$ ортогональным преобразованием(отражением), которое сохраняет метрику. Таким образом, $\ntr$ --- тоже кратчайшая замкнутая кривая, как и $l$, но уже лежащая на полусфере.

Покажем, что кривая $\ntr$ целиком лежит на $\partial\Omega_+$. Поскольку $\ntr$ --- кратчайшая кривая, то она должна быть локально кратчайшей, поэтому если $\ntr(t)\notin\partial\Omega,t\in (a;b)$, то участок $\ntr([a;b])$ должен быть геодезической, соединяющей точки $l(a)$ и $l(b)$, то есть дугой большой окружности. Отсюда следует, что $\ntr([0;1])$ состоит из частей $\partial\Omega$, соединенных друг с другом дугами больших окружностей. Но область $\Omega_+$ выпукла, поэтому кратчайшая геодезическая (меньшая дуга большой окружности), соединяющая точки $\partial\Omega_+$, лежит внутри $\overline{\Omega_+}$. Если геодезическая не кратчайшая, тогда это б\'ольшая дуга большой окружности, которая выходит за пределы полусферы. Если существует несколько геодезических, то это противоположные точки на сфере, но на $\partial\Omega_+$ нет противоположных точек. Поэтому не существует геодезических на полусфере, соединяющих различные точки $\partial\Omega_+$. Отсюда следует, что $g$ полностью лежит на $\partial\Omega_+$.

Поскольку в сферических координатах для кривой $g$ $|\phi(0)-\phi(1)|=2\pi n$, то каждый меридиан  $\phi\equiv a\mdpi$ кривая $g$ пересекает не менее $N$ раз, и через каждую точку $\partial\Omega$ проходит не менее $N$ раз, поэтому ее длина не менее $N\omlen$. поскольку эта кривая кратчайшая, то ее длина равна ровно $N\omlen$, иначе кривая, $N$ раз проходящая по границе $\Omega_+$ была бы короче ее.
 Отсюда 
 $$L(\vtr)\ge L(\tilde{l})-|\teta(t_1)-\teta(t_0)|\ge L(g)-|\teta(t_1)-\teta(t_0)|= N\omlen-|\teta(t_1)-\teta(t_0)|.$$
\end{proof}

\setcounter{theorem}{0}
\noindent\textit{Доказательство Теоремы 1.}
%\begin{theorem}
%Пусть $y$~--- решение $(*)$. Тогда $\forall t>0$ верна оценка\\
%\begin{center}
% $\ds\gamma(x,t)>\half[\nu(x,t)-5]\omlen$
%\end{center}
%\end{theorem}
\begin{proof}
Сначала будем считать, что у $y$ нет кратных нулей.
Пусть $\floorbr{\half[\nu(y,t)]}-1=N$, т.е. у функции $y$ не менее, чем $2(N+1)$ нулей. Это означает, что наша кривая $\tr(t)$ пересекает $N+1$ раз дугу $\phi=2\pi k$, а это значит, к каждому участку кривой $\tr(t)$ между 2-мя такими пересечениями можно применить лемму \ref{deletePart}, и получим кривую $\vtr$, на которой $\phi$ изменяется на ту же величину, длина которой не больше, чем длина $\tr$, но уже не пересекающуюся с $\Omega_+$, а значит удовлетворяющую лемме \ref{KLength}. Применяя ее и учитывая, что в сферической системе $\teta\in[-\half[\pi];\half[\pi]]$, получаем:
$$\gamma_3(y,t)=L(\tr, [0,t])\ge L(\vtr)\ge N\omlen - |\tilde{\teta}(t)-\tilde{\teta}(0)| \ge \rbr{\floorbr{\half[\nu(y,t)]}-1}\omlen-\pi.$$
Учитывая, что отрезок длины $\half[\pi]$ соединяет две точки $\partial\Omega_+$ ($(1,0,0)$ и $(0,0,1)$), то длина кривой не менее, чем $\pi$. Поэтому
$$\gamma_3(y,t)\ge\rbr{\floorbr{\half[\nu(y,t)]}-1}\omlen-\pi\ge \rbr{\floorbr{\half[\nu(y,t)]}-2}\omlen>\half[\nu(y,t)-5]\omlen.$$

Теперь сведем случай с кратными нулями к уже рассмотренному. Для этого заметим, что на конечном отрезке достигается минимум непрерывной функции $|x(t)|$, причем он не равен 0 (иначе решение было бы $\equiv 0$). Пусть $C:=\min\limits_{t\in[0;t_0]} |x(t)|$. 
$$\gamma(x,t_0)=\int\limits_{0}^{t_0}|\dot{\tr}(t)|\,dt=\int\limits_{0}^{t_0}\frac{\sqrt{(\dot{x},\dot{x})(x,x)+(x,\dot{x})^2}}{|x|^2}\,dt.$$
Под интегралом функция, непрерывная относительно $x$ и $\dot{x}$ при $|x|\ge C>0$, а значит равномерно непрерывна на компакте $\fbr{x: |x|+|\dot{x}|\le M, |x|\ge C>0}$. $x,\dot{x}$ ограничены на конечном отрезке, поэтому $\gamma(x,t_0)$ непрерывна по $x$ относительно нормы $\|x\|_{C^1}$, а значит $\gamma_3(y,t_0)$ непрерывна по $y$ относительно нормы $\|y\|_{C^3}$. Поэтому если взять последовательность $y_n\ra y$ в норме $C^3[0;t_0]$ такую, что $\nu(y_n,t)\ge \nu(y,t)$, то $\gamma_3(y,t_0)=\llim{n}\gamma_3(y_n,t_0)\ge \rbr{\floorbr{\half[\nu(y,t)]}-1}\omlen-\pi$. Далее будем строить такую последовательность.

Пусть $(t_i)_{i=1}^m$~--- точки, в которых $y=\dot{y}=0$. Их конечное число, иначе была бы предельная точка, и в ней $\ddot{y}=0$, но тогда $y\equiv 0$.
Положим $y_n=y+\frac{1}{n}\Delta y$, где
$$
\Array{l}{
\Delta y(t_i)=-\sign \ddot{y}(t_i),\\
\Delta \dot{y}(t_i)=\Delta\ddot{y}(t_i)=\Delta \dddot{y}(t_i)=0;\\
\Delta y(t) = -\sign\ddot{y}(t_1), t<t_1,\\
\Delta y(t) = -\sign\ddot{y}(t_m), t>t_m.
}
$$
На отрезках $[t_i;t_{i+1}]$ функция $\Delta y$ определяется многочленом степени 7 по двум точкам и трем производным в них, исходя из первых двух условий. Тогда по построению $\Delta y\in C^3$ и при достаточно большом $n$ функция $y+\frac{1}{n}\Delta y$ имеет нулей не меньше, чем $y$, но нет кратных, что и требовалось.
\end{proof}

\begin{Lemma}
$\forall\eps>0$ существуют гладкие функции $a(t),b(t),c(t)$ такие, что существует решение $y(t)$ уравнения $(*)$, для которого выполнено:
$$\nu(y)\ne 0 \mbox{ и }\frac{\mu_3(y)}{\nu(y)}<\frac{\omlen}{2\pi}+\eps.$$
\end{Lemma}
\begin{proof}
Рассмотрим такую функцию, для которой проекция фазовой кривой на сферу является замкнутой кривой, близкой к $\partial\Omega_+$. Для этого сначала зададим саму эту проекцию в сферической системе: $\teta=F(\phi)$ так, чтобы график $F$ был вне $\overline{\Omega}_+$, но его длина была $\omlen+\delta$, это можно сделать, если взять гладкую\footnote{можно взять и кусочно-гладкую $F$, но тогда коэффициенты уравнения будут разрывными по $t$} $F$, приближающую снизу функцию $\teta_0(\phi)$, задающую $\partial\Omega_+$. Для разность длин кривых $\teta=f(\phi)$ и $\teta=g(\phi)$ на сфере будет такая формула:
\begin{equation}
\oint\limits_{\gamma_f}ds-\oint\limits_{\gamma_g} ds=\int\limits_{0}^{2\pi}\rbr{\sqrt{{f'}^2+\cos^2 f}-\sqrt{{g'}^2+\cos^2 g}}\,d\phi\le \int\limits_{0}^{2\pi}(|f'-g'|+|f-g|)\,d\phi,
\label{lendiff}
\end{equation}
т.к. 
$$ds^2=d\teta^2+d\phi^2\cos^2\teta,$$
а второе неравенство выполнено, поскольку:
$$|\sqrt{a^2+b^2}-\sqrt{c^2+d^2}|=\frac{|a^2+b^2-c^2-d^2|}{\sqrt{a^2+b^2}+\sqrt{c^2+d^2}}\le
\frac{|a^2-c^2|+|b^2-d^2|}{\sqrt{a^2+b^2}+\sqrt{c^2+d^2}}\le$$
$$\le\frac{|a^2-c^2|}{|a|+|c|}+\frac{|b^2-d^2|}{|b|+|d|}\le ||a|-|c||+||b|-|d||\le |a-c|+|b-d|.$$
Пусть $F$ --- гладкая функция, которая приближает функцию $\tilde{\teta}_0-\frac{\delta}{4\pi}$ с точностью $\frac{\delta}{16\pi}$  по норме $W^1_1$, где $\tilde{\teta}_0(\phi)=\left\{\Array{ll}{\teta_0(\phi),&\phi\in(-\pi/2,\pi/2);\\\half[\pi],&\mbox{иначе.}}\right.$

Тогда везде $F(\phi)\le \tilde{\teta}_0(\phi)-\frac{\delta}{4\pi}+\frac{\delta}{8\pi}\le\half[\pi]-\frac{\delta}{8\pi}$.
Из \eqref{lendiff} получаем, что 
$$\left|\int\limits_{0}^{2\pi}F(\phi)\,d\phi-\omlen\right|=\left|\int\limits_{0}^{2\pi}F(\phi)\,d\phi-\int\limits_{0}^{2\pi}\teta_0(\phi)\,d\phi\right|\le 2\pi\frac{\delta}{4\pi}+\frac{\delta}{16\pi}<\delta.$$
%\comment{надо наверное этот момент подробнее пояснить, хотя это очевидно, если нарисовать}
Далее поскольку $(\phi,F(\phi))\notin\overline{\Omega}$, то по определению $\Omega$ это означает, что $y\ddot{y}-\dot{y}^2<0$, или, что то же самое, $\dot{\phi}=\frac{y\ddot{y}-\dot{y}^2}{y^2+\dot{y}^2}<0$. Выразим $\dot{\phi}$ через $\phi$ и $\teta$:
\begin{equation}
\dot{\phi}=\frac{y\ddot{y}-\dot{y}^2}{y^2+\dot{y}^2}=\frac{\sin\teta \cos\phi\cos\teta - (\sin\phi\cos\teta)^2}{\underbrace{(\cos\phi\cos\teta)^2+(\sin\phi\cos\teta)^2}_{\cos^2\teta}}=\tg\teta\cos\phi-\sin^2\phi.
\label{eqphi0}
\end{equation}

Учитывая, что $\teta=F(\phi)$, и переворачивая обе части равенства, получаем:
$$\frac{dt}{d\phi}=\frac{1}{\tg\teta\cos\phi-\sin^2\phi}=\frac{1}{\tg F(\phi)\cos\phi-\sin^2\phi}<0.$$
Пусть $\phi(0)=\pi$. Тогда $\ds t(\phi)=\int\limits_{\pi}^{\phi}\frac{d\phi}{\tg F(\phi)\cos\phi-\sin^2\phi}\downarrow$. Так как функция $F$ периодическая по $\phi$, с периодом $2\pi$ (т.к. ее график на сфере~--- замкнутая кривая), то $t(2\pi k+\pi)=k t(3\pi)\ra\infty,k\ra-\infty$. Отсюда и из строгой монотонности следует, что обратная функция $\phi_0(t)$ определена всюду на $[0;+\infty)$ и $\phi_0(t)\ra-\infty,t\ra+\infty$.
Найдем уравнение на функцию $y$, соответствующую такой кривой на сфере. Для этого посчитаем $\dot{\teta}$ через производные $y$. $\teta = \arctan\frac{\ddot{y}}{\sqrt{y^2+\dot{y}^2}}$, значит
$$\dot{\teta}=\frac{\rbr{\frac{\ddot{y}}{\sqrt{y^2+\dot{y}^2}}}'}{1+\frac{\ddot{y}^2}{y^2+\dot{y}^2}}=\frac{\dddot{y}\sqrt{y^2+\dot{y}^2}-\frac{
\ddot{y}(y\dot{y}+\dot{y}\ddot{y})}{\sqrt{y^2+\dot{y}^2}}}{y^2+\dot{y}^2+\ddot{y}^2}.$$
Отсюда выражаем $\dddot{y}$, учитывая, что
\newcommand{\syy}{\sqrt{y^2+\dot{y}^2}}
$\frac{y}{\syy}=\cos\phi,\frac{\dot{y}}{\syy}=\sin\phi,\frac{\ddot{y}}{\syy}=\tg\teta$:
$$\dddot{y}=\frac{\dot{\teta}(y^2+\dot{y}^2+\ddot{y}^2)}{\sqrt{y^2+\dot{y}^2}}+\frac{\ddot{y}(y\dot{y}+\dot{y}\ddot{y})}{y^2+\dot{y}^2}=\dot{\teta}\rbr{y\cos\phi+\dot{y}\sin\phi+\ddot{y}\tg\teta}+\ddot{y}\rbr{\sin\phi(\cos\phi+\tg\teta)}.$$
%\dot{\teta}\frac{y+\dot{y}\frac{\dot{y}}{y}+\ddot{y}\frac{\ddot{y}}{y}}{\sign y\sqrt{1+\frac{\dot{y}^2}{y^2}}}+\frac{\ddot{y}(\frac{\dot{y}}{y}+\frac{\dot{y}}{y}\frac{\ddot{y}}{y})}{1+\frac{\dot{y}^2}{y^2}}=$$
%$$=\dot{\teta}\cos\phi(y+\tg\phi\dot{y}+\frac{\tg\teta}{\cos\phi}\ddot{y})+\ddot{y}(\tg\phi+\tg\phi\frac{\tg\teta}{\cos\phi})\cos^2\phi=$$
Приводя подобные, получаем уравнение для $y$:
\begin{equation}
 \dddot{y}=\dot{\teta}\cos\phi\,y + \dot{\teta}\sin\phi\,\dot{y} + (\dot{\teta}\tg\teta+\sin\phi(\cos\phi+\tg\teta))\,\ddot{y}.\label{teta0eq}
\end{equation}
Учитывая, что $\phi=\phi_0(t),\teta=F(\phi_0(t))=:\Theta_0(t)$, %,\dot{\teta}=\teta'_{\phi}\dot{\phi_0}=F'(\phi)(\tg\teta\cos\phi-\sin^2\phi)$
 получаем, что это линейное уравнение с коэффициентами, гладкость которых зависит от гладкости $F$:

\begin{equation}
\Array{l}{
\dddot{y}=A(t)y+B(t)\dot{y}+C(t)\ddot{y},\\
A(t)=\dot{\Theta}_0(t)\cos\phi_0(t),\\
B(t)=\dot{\Theta}_0(t)\sin\phi_0(t),\\
C(t)=\dot{\Theta}_0(t)\tg\Theta_0(t)+\sin\phi_0(t)(\cos\phi_0(t)+\tg\Theta_0(t)).\\
}\label{th2eq}
\end{equation}

$|F(x)|<\half[\pi]$, поэтому $\tg\Theta_0(t)=\tg F(\phi_0(t))$ определен и непрерывен всюду. 
Проверим, что есть решение этого уравнения, для которого след на сфере будет иметь сферические координаты $(\phi_0,\teta_0)$.
Возьмем $y_0$~--- решение уравнения \eqref{th2eq} с начальными данными: $y(0)=-1,\dot{y}(0)=0,\ddot{y}(0)=\tg F(\pi)$, тогда $\phi(0)\equiv\pi\equiv\phi_0(0)\mdpi$~--- согласуется условием на функцию $\phi_0(t)$.
% также на решении выполнено и уравнение (\ref{eqphi0}). $\teta=F(\phi)$. Из единственности решения задачи Коши для функции $\phi$ и $\teta$. 
Исходя из равенств~\eqref{eqphi0}~и~\eqref{teta0eq} и начальных условий, для углов сферической системы координат будут выполняться уравнения: 
$$\left\{\Array{ll}
{\ds\dot{\phi}=\tg\teta\cos\phi-\sin^2\phi, & \phi(0)=\pi;\\
\ds\dot{\teta}=\frac{\dddot{y_0}-\ddot{y_0}\sin\phi(\cos\phi+\tg\teta)}{ y_0\cos\phi+\dot{y_0}\sin\phi+\ddot{y_0}\tg\teta},&\teta(0)=F(\pi).}\right.$$
Те же уравнения будут выполняться для функций $\phi_0$ и $\Theta_0$ по построению уравнения~\eqref{th2eq}. Заметим, что на решении знаменатель правой части второго уравнения будет иметь вид: 
$$y_0\cos\phi+\dot{y_0}\sin\phi+\ddot{y_0}\tg\teta=\frac{y_0^2+\dot{y}_0^2+\ddot{y}_0^2}{y_0^2+\dot{y}_0^2},$$
и нигде не обращается в 0. Поэтому в окрестности решения правые части уравнения \eqref{th2eq} непрерывны, поэтому оно единственно. Из единственности решения задачи Коши, получим, что $\phi(t)=\phi_0(t), \theta(t)=\Theta_0(t)$.

 Посчитаем $\frac{\mu_3(y)}{\nu(y)}$.
%Найдем саму функцию $y$: $\frac{\dot{y}}{y}=\tg\phi(t)\ \Ra\ y=Ce^{\int\limits_{0}^{t}\tg\phi(t)\,dt}$
Пусть $T:=-t(3\pi)$. Так как $T$~--- период проекции фазовой кривой на сфере, то $\gamma_3(y,nT)=n\gamma_3(y,T)\le n(\omlen+\frac{3}{4}\delta)$.
Поэтому 
\newcommand{\ffl}{\floorbr{\frac{t}{T}}}
$$\mu_3(y)=\llim{t}\frac{\gamma_3(y,t)}{t}=\llim{t}\frac{\gamma_3(y,\ffl T)+O(1)}{\ffl T+O(1)}\le\llim{t}\frac{\ffl (\omlen+\delta)}{\ffl T}=\frac{\omlen+\delta}{T};$$
$$\nu(y)=\llim{t}\frac{\pi\nu(y,t)}{t}=\llim{t}\frac{\pi\nu(y,\ffl T)+O(1)}{\ffl T+O(1)}=\llim{t}\frac{2\pi\ffl + O(1)}{\ffl T}=\frac{2\pi}{T}.$$
Здесь мы учли, что на каждом периоде длиной $T$ есть ровно 2 корня: при $\phi\equiv\half[\pi]\mdpi$ и $\phi\equiv-\half[\pi]\mdpi$. Поэтому $\nu(y,nT)=2n$.
$$\frac{\mu_3(y)}{\nu(y)}<\frac{\omlen+\delta}{2\pi}.$$
Выберем $\delta<2\pi\eps$, и утверждение основной леммы будет выполнено.
\end{proof}

%\setcounter{theorem}{0}
%\begin{theorem}
%$\ds\hat{\mu}(x)\ge\frac{\omlen}{2\pi}\hat{\nu}(x)$ % $\check{\mu}(x)\ge\frac{\omlen}{2\pi}\check{\nu}(x)$ 
%и $\forall\eps>0$ существует такое $x$, что $\hat{\mu}(x)<\rbr{\frac{\omlen}{2\pi}+\eps}\hat{\nu}(x)$. Аналогичное утверждение верно для $\check{\mu}$ и $\check{\nu}$.
%\end{theorem}
\noindent \textit{Доказательство теоремы 2.}
\begin{proof}
Докажем, что это нижняя оценка. Сначала для первого отношения:\\
будем считать, что $\hat{\nu}\ne 0$. Тогда $\ds\exists \{t_n\} : \hat{\nu}(x)=\lim\limits_{n\ra\infty}\frac{\pi}{t_n}\nu(x,t_n)$.
По свойству верхнего предела, и подставляя оценку из основной леммы 1, получаем требуемое: $$\ds\hat{\mu}(x)=\uplim_{t\ra\infty}\frac{\gamma(x,t)}{t}\ge \upllim{n}\frac{\gamma(x,t_n)}{t_n}\ge \upllim{n}\frac{1}{t_n}\half[\nu(x,t_n)-5]\omlen =$$ $$=\omlen\upllim{n}\rbr{\frac{\nu(x,t_n)}{2t_n}+\frac{O(1)}{t_n}}=\omlen\upllim{n}\frac{\nu(x,t_n)}{2t_n}=\frac{\omlen}{2\pi}\hat{\nu}(x).$$

Для второго соотношения все аналогично, только нужно взять \\$\ds{t_n}: \check{\mu}(x)=\llim{n}\frac{\gamma(x,t_n)}{t_n}$. По теореме 1 $\ds\nu(x,t)<\frac{2\gamma(x,t)}{\omlen}+5$, откуда можно вывести требуемое:
$$\check{\nu}(x)=\downllim{t}\frac{\pi}{t}\nu(x,t)\le \downllim{n}\frac{\pi}{t_n}\nu(x,t_n)\le \pi\downllim{n}\frac{2\gamma(x,t_n)+5}{t_n\omlen} = $$
$$=\frac{2\pi}{\omlen}\llim{n}\frac{\gamma(x,t_n)+O(1)}{t_n}=\frac{2\pi}{\omlen}\check{\mu}.$$
Из основной леммы 1 следует, что $\frac{\mu(x)}{\nu(x)}$ может быть сколь угодно близко к $\frac{\omlen}{2\pi}$, поэтому оценку нельзя улучшить.
\end{proof}

\subsection*{Добавление: вычисление постоянной $L$}
Было показано, что $\partial\Omega_+$ в сферических координатах задается формулой 
$$\teta=\teta_0(\phi)=\arctan\frac{\sin^2 \phi}{\cos\phi},\quad \phi\in\rbr{-\half[\pi],\half[\pi]}.$$
Поэтому, получаем: 
$$ds^2=\cos^2\teta \,d\phi^2+d\teta^2=(\cos^2\teta_0(\phi)+\teta_0'^2(\phi))\,d\phi^2=\rbr{\frac{1}{1+\tg^2\teta_0}+\rbr{\dfrac{\rbr{\frac{\sin^2\phi}{\cos{\phi}}}'}{1+\frac{\sin^4\phi}{\cos^2\phi}}}^2}\,d\phi^2=$$
$$=\rbr{\frac{\cos^2\phi}{\cos^2\phi+\sin^4\phi}+\frac{(2\sin\phi\cos^2\phi+\sin^3\phi)^2}{(\cos^2\phi+\sin^4\phi)^2}}\,d\phi^2=$$
$$=\frac{\cos^2\phi(\cos^2\phi+\sin^4\phi)+\sin^2\phi(2\cos^2\phi+\sin^2\phi)^2}{(\cos^2\phi+\sin^4\phi)^2}\,d\phi^2.$$
Преобразуем по отдельности числитель и знаменатель полученного выражения:
$$\cos^2\phi(\cos^2\phi+\sin^4\phi)+\sin^2\phi(2\cos^2\phi+\sin^2\phi)^2=\cos^4\phi+\cos^2\phi\sin^4\phi+\sin^2\phi(1+\cos^2\phi)^2=$$
$$=\sin^2\phi+2\sin^2\phi\cos^2\phi+\sin^2\phi\cos^4\phi+\cos^4\phi+\cos^2\phi\sin^4\phi=\sin^2\phi+\frac{\sin^2 2\phi}{2}+$$
$$+\cos^2\phi(\cos^2\phi+\sin^2\phi(\sin^2\phi+\cos^2\phi))=1+\half[\sin^2 2\phi]=1+\frac{1-\cos 4\phi}{4}=\frac{5-\cos 4\phi}{4};$$
$$\cos^2\phi+\sin^4\phi=\cos^2\phi+\sin^2\phi(1-\cos^2\phi)=1-\sin^2\phi\cos^2\phi=\frac{4-\sin^2 2\phi}{4}=\frac{7+\cos 4\phi}{8}.$$
%$$(\cos^2\phi+\sin^4\phi)^2=\frac{4-\sin^2 2\phi}{}$$
Итак, получаем:
$$ds^2= \frac{5-\cos 4\phi}{4\rbr{\frac{7+\cos 4\phi}{8}}^2}\,d\phi^2=16\frac{5-\cos 4\phi}{(7+\cos 4\phi)^2}\,d\phi^2.$$
Отсюда находим длину кривой:
$$L=\int\limits_{-\half[\pi]}^{\half[\pi]}4\frac{\sqrt{5-\cos 4\phi}}{7+\cos 4\phi}\,d\phi=\left<\alpha=4\phi\right>=\int\limits_{-2\pi}^{2\pi}\frac{\sqrt{5-\cos\alpha}}{7+\cos\alpha}\,d\alpha=4\int\limits_{0}^{\pi}\frac{\sqrt{5-\cos\alpha}}{7+\cos\alpha}\,d\alpha$$
В последнем равенстве мы учли периодичность и четность косинуса.

\end{document}